\documentclass[oneside,english]{siamltex}
\usepackage[T1]{fontenc}
\usepackage[latin9]{inputenc}
\usepackage{fancyhdr}
\usepackage{babel}
\usepackage{units}
\usepackage{amssymb}
\usepackage{graphicx}
\usepackage{esint}
\usepackage[unicode=true,pdfusetitle,
 bookmarks=true,bookmarksnumbered=false,bookmarksopen=false,
 breaklinks=true,pdfborder={0 0 0},backref=false,colorlinks=false]
 {hyperref}
\usepackage{breakurl}

\makeatletter
\newcommand\eqref[1]{(\ref{#1})}

\usepackage{fancyvrb}
\makeatother

\begin{document}

\title{Eigenvalues for Two-Lag Linear Delay Differential Equations%
\thanks{Approved for Public Release, Distribution Unlimited.%
}}

\author{D. M. Bortz\thanks{Mathematical Biology Group, Applied Mathematics, University of Colorado, Boulder, CO 80309-0526}}
\maketitle
\begin{abstract}
We consider the class of two-lag linear delay differential equations
and develop a series expansion to solve for the roots of the nonlinear
characteristic equation. Supporting numerical results are presented
along with application of our method to study the stability of a two-lag
model from ecology.\end{abstract}

\section{\label{sec:Introduction}Introduction}

Delay Differential Equations (DDEs) naturally arise in many fields
of science and engineering%
\footnote{See Erneux \cite{Erneux2009} for a good overview of different application
areas.%
} and accordingly have been the focus of extensive study over the years.
The classic books by Bellman and Cooke \cite{Bellman1963} and El'sgol'ts
and Norkin \cite{Elsgolts1973} describe the significant mathematical
results up to the 1960's, while subsequent books by Hale and Verduyn
Lunel \cite{Hale1993} and Diekmann et al., \cite{Diekmann1995} cover
the later semigroup-based frameworks. A specific area in which delays
arise frequently is biological modeling, e.g., due to gestation periods,
and there are several good texts on population modeling and analysis
with DDEs \cite{Cushing1977,Gopalsamy1992,Kuang1993,MacDonald1989}
including a recent introductory textbook by Smith \cite{Smith2011a}.
Lastly, we direct the interested reader to several recent review articles
\cite{Gu2003a,Richard2003} and edited works \cite{Balachandran2009,Loiseau2009,Sipahi2012a}
which thoroughly cover the state-of-the-art advances in the study
and application of DDEs.

\subsection{DDE Stability Analyses}

Consider the first order, linear, single-lag DDE
\begin{eqnarray}
\dot{x}(t) & = & \alpha x(t)+\beta x(t-\tau)\,;\, t>0\,,\label{eq:DDE eqn}\\
x(t) & = & \phi(t)\,;\, t\leq0\,,\label{eq:DDE hist}
\end{eqnarray}
where $\phi(t)$ is the initial history function (which need not satisfy
(\ref{eq:DDE eqn})). The Laplace transform of (\ref{eq:DDE eqn})
is 
\begin{equation}
s=\alpha+\beta e^{-s\tau}\,,\label{eq:nonline eval eqn 1}
\end{equation}
a transcendental equation with an infinite number of roots $\{s_{j}\}$
in the complex plane (see Figure \ref{fig:computed eval roots} for
an example). Thus the solutions are constructed as an infinite series
\begin{eqnarray*}
y(t) & = & \sum_{j=-\infty}^{\infty}C_{j}e^{s_{j}t}\,,
\end{eqnarray*}
and we note that these basis functions $e^{s_{j}t}$ form a Schauder
basis for $L^{p}$ functions over a finite domain. The coefficients
$C_{j}$ can be computed via evaluating the history function $\phi(t)$
and the basis functions $e^{s_{j}t}$ at $2N+1$ timepoints in $[-\tau,0]$
and then linearly solving for $2N+1$ of the $C_{j}$'s.

The nonlinear eigenvalue equation in (\ref{eq:nonline eval eqn 1})
belongs to a well-known class of equations known as \emph{exponential
polynomials} \cite{Bell1934,Moreno1973,Ritt1929} and many researchers
in DDEs have (understandably) chosen to study them \cite{Avellar1980,Cooke1986}.
In the broader mathematics community, research into this class of
algebraic equations is very active and extends to numerical schemes
for computing the roots (see papers by Jarlebring such as \cite{Jarlebring2009})
and applications to fields afar as quantum computing \cite{Sasaki2009}.

A primary goal in the study of linear (and linearized) DDEs is to
identify the value of the the delay $\tau$ for which equilibria become
unstable, i.e., when the real part of the principal root $\mbox{Re}(s_{0})$
becomes positive. A traditional strategy for proving that such a $\tau$
exists is to separate (\ref{eq:nonline eval eqn 1}) into real and
imaginary equations and prove for the $\tau_{*}$ such that $s_{0}(\tau_{*})$
is purely imaginary, the eigenvalue passes through to the positive
real half of the plane, i.e., 
\[
\left.\frac{d}{d\tau}\mbox{Re}(s_{0}(\tau))\right|_{\tau=\tau_{*}}>0\,.
\]
 Many researchers have been very successful in establishing stability
results using this framework \cite{Forde2004}, even for system with
distributed delays \cite{Cooke1996,Yuan2011}, multiple delays \cite{Braddock1983,Li1999,Ruan2003},
and fractional order derivatives \cite{Deng2006}.

\subsection{Relationship between DDEs and Lambert W function}

In their 2003 article, Asl and Ulsoy \cite{Asl2003} connected the
roots of (\ref{eq:nonline eval eqn 1}) with solutions to a specific
exponential polynomial known as the \emph{Lambert W} equation \cite{Corless1996,Corless1997}.
Briefly: to solve for the roots to (\ref{eq:nonline eval eqn 1}),
rewrite the equation in the form 
\[
\tau(s-\alpha)e^{\tau(s-\alpha)}=\beta\tau e^{-\alpha\tau}\,,
\]
and define a new function $W$ such that $W(\beta\tau e^{-\alpha\tau})=\tau(s-\alpha)$.
The equation to be solved is now 
\begin{equation}
W(\beta\tau e^{-\alpha\tau})e^{W(\beta\tau e^{-\alpha\tau})}=\beta\tau e^{-\alpha\tau}\,,\label{eq:nonlin eval eqn with W}
\end{equation}
which is in the same form as $W(x)e^{W(x)}=x\,,$ the well-studied
Lambert W equation. The roots to (\ref{eq:nonline eval eqn 1}) are
thus 
\[
s=\frac{1}{\tau}W(\beta\tau e^{-\alpha\tau})+\alpha\,.
\]

Caratheodory provided the principal solution to the Lambert W function
\[
W_{0}(x)=\sum_{n=1}^{\infty}\frac{(-n)^{n-1}}{n!}x^{n}\,,
\]
and all other branches have been computed as well
\begin{equation}
W_{j}(x)=\ln(x)+2\pi ij-\ln(\ln(x)+2\pi ij)+\sum_{l=0}^{\infty}\sum_{m=1}^{\infty}C_{lm}\frac{(\ln(\ln(x)+2\pi ij))^{m}}{(\ln(x)+2\pi ikj)^{l+m}}\,,\label{eq:wk series}
\end{equation}
with coefficients 
\[
C_{lm}=\frac{1}{m!}(-1)^{l}{l+m \brack l+1}\,.
\]
expressed in terms of the unsigned Stirling numbers of the first kind
\cite{Abramowitz1972}. This expansion to compute the different roots
for the Lambert W function was presented by de~Bruijn \cite{Bruijn1958},
with Comtet \cite{Comtet1974} rewriting de~Bruijn's expansion in
terms of the Stirling numbers and Corless et al.~\cite{Corless1996}
correcting a typographic error in equation (2.4.4) in Comtet \cite{Comtet1974}.
Lastly, we note the \textbf{lambertw }function in Matlab, \textbf{LambertW
}in Maple, and \textbf{ProductLog }in Mathematica allow for the numerically
accurate computation of $W_{j}(x)$ for any branch $j$.

Since Asl and Ulsoy's original article, the connection between DDE
eigenvalues and solutions to Lambert W has been substantially extended
and expanded \cite{Chen2002,Chen2002a,Chen2003,Cheng2006,Cogan2011,Hwang2005,Hwang2005a,Jarlebring2007,Shinozaki2006,Wang2008,Yi2007,Yi2008,Yi2006,Yi2006b},
culminating in a monograph \cite{Yi2010}. These results, however,
do not extend to multi-lag DDEs as (to the best of our knowledge)
the nonlinear eigenvalue problem can be cast as a Lambert W equation
\emph{only} when the DDE has a \emph{single} lag.

In this work, we develop a series expansion for the roots to the nonlinear
eigenvalue problem corresponding to a \emph{two}-lag DDE. While we
were inspired by the original derivation of (\ref{eq:wk series}),
the development presented below is a bit more involved than the Lambert
W derivation in de Bruijn \cite{Bruijn1958} partially due to our
goal of making the truncated sums computationally efficient to evaluate.

Section \ref{sec:Two-Lag} presents the main result of the paper,
deriving the formula for the eigenvalues. Section \ref{sec:Numerical-Results}
provides numerical results supporting our claims of accuracy and convergence
and including an application to a two-lag model from ecology. Lastly,
Section \ref{sec:Conclusion-and-Discussion} summarizes the results
of the work and discusses the natural next steps for future work.

\section{\label{sec:Two-Lag}Two-Lag Development}

In this section, we develop our series expansion for the roots to
the nonlinear characteristic equation arising from the two-lag DDE.
Note that for the multi-delay case, the conventional approach focuses
on developing bounds for the roots (see \cite[\S 12]{Bellman1963}
for a good summary and overview), whereas here, we are explicitly
computing the roots.

Consider a DDE with two lags $\tau_{1}$ and $\tau_{2}$
\[
\dot{x}(t)=\alpha x(t)+\beta x(t-\tau_{1})+\gamma x(t-\tau_{2})\,,
\]
 with corresponding coefficients $\alpha$, $\beta$, and $\gamma$.
We will rescale and rearrange terms, defining several new variables
and parameters in order to reduce the equation to its essential mathematical
features. We rescale time by the first lag $\mathfrak{t}=t/\tau_{1}$,
the parameters so that $\alpha_{1}=\alpha\tau_{1}$, $\beta_{1}=\beta\tau_{1}$,
$\gamma_{1}=\gamma\tau_{1}$ , $\tau=\nicefrac{\tau_{1}}{\tau_{2}}$,
and denote $x'(\mathfrak{t})=\frac{d}{d\mathfrak{t}}x(\mathfrak{t})=\tau_{1}\frac{d}{dt}x(t)$
to obtain the equation 
\[
x'(\mathfrak{t})=\alpha_{1}x(\mathfrak{t})+\beta_{1}x(\mathfrak{t}-1)+\gamma_{1}x(\mathfrak{t}-\tau)\,.
\]
 After a Laplace transform, the corresponding nonlinear characteristic
equation for $s\in\mathbb{C}$ is thus
\begin{equation}
s=\alpha_{1}+\beta_{1}e^{-s}+\gamma_{1}e^{-s\tau}\,.\label{eq:schar}
\end{equation}
We then let $\lambda=s-\alpha_{1}$, $\beta_{2}=\beta_{1}e^{-\alpha_{1}}$,
$\gamma_{2}=\gamma_{1}e^{-\alpha_{1}\tau}$, and multiply through
by $e^{\lambda\tau}$, yielding the much cleaner equation
\begin{equation}
\lambda e^{\lambda\tau}=\beta_{2}e^{(\tau-1)\lambda}+\gamma_{2}\,.\label{eq:lamchar}
\end{equation}

Inspired by the original expansion in \cite{Bruijn1958}, we note
that for $\left|\gamma_{2}\right|$ near to or far from zero, we can
write $\lambda$ as a small perturbation $u$ of $\ln\gamma_{2}$
\[
\lambda=(\ln\gamma_{2}+u)/\tau\,,
\]
 and then safely assume that
\begin{equation}
\left|u\right|\ll\left|\ln\gamma_{2}\right|\,.\label{eq:assump 1}
\end{equation}
Substituting this form of $\lambda$ into equation (\ref{eq:lamchar})
we find that
\begin{equation}
\frac{1}{\tau}\left(1+\frac{u}{\ln\gamma_{2}}\right)\gamma_{2}e^{u}=\frac{\beta_{2}\gamma_{2}^{\left(\frac{\tau-1}{\tau}\right)}}{\ln\gamma_{2}}\left(e^{u}\right)^{\left(\frac{\tau-1}{\tau}\right)}+\frac{\gamma_{2}}{\ln\gamma_{2}}\,.\label{eq:lamchar2}
\end{equation}
To proceed, we will also assume that 
\begin{equation}
\left|\frac{\beta_{2}}{\gamma_{2}^{1/\tau}\ln\gamma_{2}}\right|\ll1\,,\label{eq:assump 2}
\end{equation}
in addition to (\ref{eq:assump 1}). This restriction on the coefficients
is a result of the chosen rearrangement of (\ref{eq:schar}) into
(\ref{eq:lamchar}). There is nothing particularly special about (\ref{eq:assump 2})
and choosing to recast (\ref{eq:schar}) differently would simply
necessitate different assumptions on the parameters in order for an
expansion to converge. In our experience, this is a reasonable assumption
encountered in several models (see the examples in Section \ref{sec:Numerical-Results}).

Moving forward, we can solve (\ref{eq:lamchar2}) for $u$ as 
\[
u\approx\ln\tau+\ln(\nicefrac{1}{\ln\gamma_{2}})
\]
which means that 
\begin{equation}
u=\ln\tau+\ln(\nicefrac{1}{\ln\gamma_{2}})+v\label{eq:lnblah v}
\end{equation}
where our attention now focuses on solving for a new variable $v$.
As discussed by Corless et al.~\cite{Corless1996}, the nested logarithms
in (\ref{eq:lnblah v}) need not select the same branch. As Corless
et al.~do, we will let the outer logarithm be the principal branch
and denote $\mbox{Ln}$ as the inner logarithm for which we have not
yet chosen a branch, i.e., 
\[
u=\ln\tau+\ln(\nicefrac{1}{\mbox{Ln}\gamma_{2}})+v\,.
\]
By substituting $u$ back into $\eqref{eq:lamchar2},$ we can cancel
several terms and simplify our problem to one of finding the roots
$v$ to the equation
\[
1+\frac{\ln\tau}{\mbox{Ln}\gamma_{2}}+\frac{\ln(\nicefrac{1}{\mbox{Ln}\gamma_{2}})}{\mbox{Ln}\gamma_{2}}+\frac{1}{\mbox{Ln}\gamma_{2}}v=\frac{\beta_{2}}{\gamma_{2}}\left(\frac{\gamma_{2}\tau}{\mbox{Ln}\gamma_{2}}\right)^{\left(\frac{\tau-1}{\tau}\right)}e^{-\nicefrac{v}{\tau}}+e^{-v}
\]
If we let
\[
\sigma=\frac{1}{\mbox{Ln}\gamma_{2}}\,,\, c=\left(\frac{\beta_{2}}{\gamma_{2}}\right)\left(\frac{\gamma_{2}\tau}{\mbox{Ln}\gamma_{2}}\right)^{\left(\frac{\tau-1}{\tau}\right)}\,,\,\mu=\frac{\ln\tau\ln(\nicefrac{1}{\mbox{Ln}\gamma_{2}})}{\mbox{Ln}\gamma_{2}}-c\,,
\]
then our efforts will focus on identifying the roots of 
\begin{equation}
f(v)\equiv e^{-v}+ce^{-v/\tau}-\sigma v-1-c-\mu\,.\label{eq:f}
\end{equation}
Note that the motivation for adding and subtracting $c$ will become
self-evident in the proof for the following lemma.
\begin{lemma}
If there exists positive numbers $a$ and $b$ such that $\left|\sigma\right|<a$
and $\left|\mu\right|<a$ then (\ref{eq:f}) has a single root inside
the region $\left|v\right|\leq b$.\end{lemma}
\begin{proof}
Let $\delta$ denote the lower bound of $\left|e^{-v}-1\right|$ on
$\left|v\right|=b$ and $\delta_{\tau}$ denote the lower bound of
$\left|ce^{-v/\tau}-c\right|$ on $\left|v\right|=b$. We can arbitrarily
choose $b=\pi\min\{1,\tau\}$ since $e^{-v}-1$ has only one root
at zero for $b\in(0,2\pi)$ and $c(e^{-v/\tau}-1)$ also has only
one root at $v=0$ for $b\in(0,2\pi\tau)$. Thus there is only one
root at $v=0$ for $e^{-v}-1+c(e^{-v/\tau}-1)$ in the region $\left|v\right|\leq\pi\min\{1,\tau\}$
and that on $\left|v\right|=\pi\min\{1,\tau\}$ we have a lower bound
of $\min\{\delta,\delta_{\tau}\}$.

Let $a=\min\{\delta,\delta_{\tau}\}/(2(\pi+1))$ and then for $\left|\sigma\right|<a$,
$\left|\mu\right|<a$, and $\left|v\right|=\pi\min\{1,\tau\}$ we
have that 
\[
\left|\sigma v+\mu\right|\leq\left|\sigma\right|\left|v\right|+\left|\mu\right|\leq\frac{\min\{\delta,\delta_{\tau}\}}{2(\pi+1)}\cdot\pi+\frac{\min\{\delta,\delta_{\tau}\}}{2(\pi+1)}=\frac{\min\{\delta,\delta_{\tau}\}}{2}<\min\{\delta,\delta_{\tau}\}\,.
\]
Since $\left|e^{-v}-1+ce^{-v/\tau}-c\right|>\min\{\delta,\delta_{\tau}\}$
on the curve defined by $\left|v\right|=\pi\min\{1,\tau\}$, we have
satisfied the hypotheses of Rouch\'e's Theorem \cite{Brown1996}
and can therefore conclude that there is only one root to (\ref{eq:f})
inside $\left|v\right|=\pi\min\{1,\tau\}$.
\end{proof}

By Cauchy's theorem, we can then compute $v$ using the well-known
formula from complex analysis
\begin{equation}
v=\frac{1}{2\pi i}\oint_{\left|\zeta\right|=b}\frac{f'(\zeta)\zeta}{f(\zeta)}\mbox{d}\zeta=\frac{1}{2\pi i}\oint_{\left|\zeta\right|=b}\frac{(-e^{-\zeta}-ce^{-\zeta/\tau}/\tau-\sigma)\zeta}{e^{-\zeta}+ce^{-\zeta/\tau}-\sigma\zeta-1-c-\mu}\mbox{d}\zeta\,,\label{eq:contourintegral}
\end{equation}
 for $b=\pi\min\{1,\tau\}$. The $1/f(\zeta)$ can be expanded into
the following (absolutely and uniformly convergent) power series
\begin{equation}
\frac{1}{f(\zeta)}=\sum_{k=0}^{\infty}\sum_{m=0}^{\infty}(e^{-\zeta}+ce^{-\zeta/\tau}-1-c)^{-k-m-1}\zeta^{k}\sigma^{k}\mu^{m}a_{m}^{m+k}\,,\label{eq:finv series}
\end{equation}
with some coefficients $a_{m}$. The substitution of (\ref{eq:finv series})
into (\ref{eq:contourintegral}) and evaluation of the contour integral
generates a double series in $m$ and $k$. Note, however, that the
$m=0$ term is zero because the integrand has the same number of roots
(at $\zeta=0$) in the numerator and the denominator. The doubly infinite
and absolutely convergent power series for $v$ in $\sigma$ and $\mu$
is thus 
\begin{equation}
v=\sum_{k=0}^{\infty}\sum_{m=1}^{\infty}a_{km}\sigma^{k}\mu^{m}\,,\label{eq:vexpansion1}
\end{equation}
 where the $a_{km}$ are coefficients independent of $\sigma$ and
$\mu$. The coefficients $a_{km}$ are computed by first computing
a power series expansion in $\sigma$ and then applying the Lagrange
Inversion Theorem \cite{Abramowitz1972} to obtain $v$ as a function
of $\mu$. To facilitate rearrangement of this series into the form
of (\ref{eq:vexpansion1}), we denote 
\[
f_{2}(w)\equiv f(w)-f(0)\,,
\]
and note that the derivative of $f(w)$ with respect to $\sigma$
is $-w$. Equation (\ref{eq:vexpansion1}) can thus be written (after
the power series expansion in $\sigma$ around zero) as 
\begin{equation}
v=\sum_{k=0}^{\infty}\sum_{m=1}^{\infty}m^{\bar{k}}h_{m,k}\frac{\mu^{m}}{m!}\frac{\sigma^{k}}{k!}\,,\label{eq:veqnsum}
\end{equation}
where $m^{\bar{k}}$ is a rising factorial%
\footnote{A rising factorial $m^{\bar{k}}$ is defined as $m(m+1)\cdots(m+k-1)$
and a falling factorial $m^{\underline{k}}$ is defined as $m(m-1)\cdots(m-k+1)$.
We note that in the combinatorics literature, $m^{(k)}$ is more commonly
used as a rising factorial (also known as the Pochhammer symbol) and
$(m)_{k}$ as a falling factorial. Here we choose the notation {}``$m^{\bar{k}}$''
for rising factorial to prevent confusion later on when using the
superscript/parentheses notation for differentiation. If it also the
notation advocated by Knuth \cite{Knuth1992}.%
}, and $h_{m,k}=\lim_{w\to0}h_{m,k}(w)$ with 
\[
h_{m,k}(w)=\frac{d^{m-1}}{dw^{m-1}}\left(w^{m+k}f_{2}(w)^{-(m+k)}\right)\,.
\]
Straightforward application of the product rule yields that 
\[
h_{m,k}(w)=\sum_{l=0}^{m-1}{m-1 \choose l}\left(\frac{d^{m-1-l}}{dw^{m-1-l}}w^{m+k}\right)\left(\frac{d^{l}}{dw^{l}}f_{2}(w)^{-(m+k)}\right)\,.
\]
 The $l$th derivative of $f_{2}(w)^{-(m+k)}$ is computed using Fa\`a
di Bruno's formula \cite{Johnson2002} for higher derivatives of composite
functions and thus 
\[
h_{m,k}(w)=\sum_{l=0}^{m-1}a_{klm}w^{k+l+1}\sum_{p=0}^{l}b_{kmp}f_{2}(w)^{-(m+k+p)}\mathbb{B}_{l,p}(f_{2}'(w),\ldots,f_{2}^{(l-p+1)}(w))\,,
\]
where 
\begin{eqnarray*}
a_{klm}={m-1 \choose l}\frac{(m+k)!}{(k+l+1)!} & \,,\, & b_{kmp}=(-1)^{p}(m+k)^{\bar{p}}\,,
\end{eqnarray*}
and $\mathbb{B}_{l,p}$ are the partial Bell polynomials \cite{Aldrovandi2001,Bell1927,Bell1934,Wheeler1987}.
Recall that for an arbitrary sequence $\{x_{i}\}$ and $l,p\in\mathbb{N}$,
the partial Bell polynomials are
\[
\mathbb{B}_{l,p}(x_{1},\ldots,x_{l-p+1})=\sum\frac{l!}{j_{1}!j_{2}!\cdots j_{l-p+1}!}\left(\frac{x_{1}}{1!}\right)^{j_{1}}\left(\frac{x_{2}}{2!}\right)^{j_{2}}\cdots\left(\frac{x_{l-p+1}}{(l-p+1)!}\right)^{j_{l-p+1}}\,,
\]
with the sum taken over all sequences $\{j_{1},\ldots,j_{n-k+1}\}\in\mathbb{N}$
such that $\sum_{i=1}^{n-k+1}j_{i}=m$ and $\sum_{i=1}^{n-k+1}ij_{i}=k$
(a simple Diophantine linear system). The sets of indices $\{j_{i}\}$
satisfying these sums are just a way of describing all possible partitions
of a set of size $k$ \cite{Encinas2005}.

In what follows, the argument for the Bell polynomial frequently takes
the form $\{f_{2}^{(1+i)}(0)\}_{i=1}^{n}$. To improve the clarity
of the formulas, we denote 
\[
F_{n}=\{f_{2}^{(1+i)}(0)\}_{i=1}^{n}\,;\, n\in\mathbb{N},
\]
 to be a finite sequence of higher derivatives of $f_{2}$.

Formally taking the limit for $\lim_{w\to0}h_{m,k}(w)$ yields 
\begin{equation}
\frac{\sum_{l=0}^{m-1}\tilde{a}_{klm}\sum_{p=0}^{l}\sum_{q=0}^{2m-2-l}\tilde{b}_{klmpq}\mathbb{B}_{2m-2-l-q,m-1-p}(F_{m-l-q+p})\mathbb{B}_{l,p}^{(q)}(F_{l-p+1})}{f_{2}^{''}(0)^{2m+k-1}}\,,\label{eq:hmkw}
\end{equation}
where 
\begin{eqnarray*}
\tilde{a}_{klm} & = & a_{klm}{2m+k-1 \choose 2m-l-2}(k+1+l)!\,,\\
\tilde{b}_{klmpq} & = & b_{kmp}{2m-2-l \choose q}\frac{(m-1-p)!}{(2m+k-1)!}\,,\,.
\end{eqnarray*}

The evaluation of most terms in (\ref{eq:hmkw}) is computationally
straightforward and efficient, except the term involving $\mathbb{B}_{l,p}^{(q)}$.
For example, the sum in (\ref{eq:veqnsum}) for $m=10$ and $k=100$
can take around one minute on an Intel i7 cpu when using Mathematica
to analytically find the $q$th derivative of $\mathbb{B}_{l,p}$
and then sum the terms. In \cite{Mishkov2000}, Mishkov derived a
formula for $\mathbb{B}_{l,p}^{(q)}$, but to implement it would consists
of solving a cascade of linear Diophantine equations%
\footnote{These have no closed form solution and would thus also need to be
solved.%
} to obtain the correct indices for the sums. 

We will use a recursive property of $\mathbb{B}_{l,p}^{(k)}$ to avoid
having to solve systems of Diophantine equations. Note that for an
arbitrary sequence $X=\{x_{j}\}$ 
\[
\frac{\partial}{\partial x_{j}}\mathbb{B}_{l,p}(X)={l \choose j}\mathbb{B}_{l-j,p-1}(X)\,,
\]
and thus in the context of (\ref{eq:hmkw}) we have that 
\begin{equation}
\mathbb{B}_{l,p}^{(q)}(F_{l-p+1})=\sum_{r_{1}=1}^{l-p+1}{l \choose r_{1}}\sum_{r_{2}=0}^{q-1}{q-1 \choose r_{2}}\mathbb{B}_{l-r_{1},p-1}^{(q-r_{2}-1)}(F_{l-p+1})\,,\label{eq:ddq Blp}
\end{equation}
recursively allows (eventual) computation of the derivatives by evaluation
of the partial Bell polynomials themselves.%
\footnote{For  a good summary of known properties of Bell polynomials, we direct
the interested reader to Aldrovandi \cite[{\S}13 \& Appendix A.5]{Aldrovandi2001}.%
} The use of Mathematica's memoization features for recursive sums
substantially sped up the computation as well, naturally at the expense
of increased memory usage.%
\footnote{Memoization is a technique for reducing the number of function calls
by storing the results of previous calls \cite{Michie1968}. %
}

We now have that the principal branch of the roots of (\ref{eq:lamchar})
is 
\begin{equation}
\lambda=(\ln\gamma_{2}+\ln\tau+\ln(\frac{1}{\ln\gamma_{2}})+v)/\tau\,,\label{eq:lambda principle branch}
\end{equation}
where 

\begin{equation}
v=\sum_{k=0}^{\infty}\sum_{m=1}^{\infty}m^{\bar{k}}h_{m,k}\frac{\mu^{m}}{m!}\frac{\sigma^{k}}{k!}\,.\label{eq:veqnsum-1}
\end{equation}
We note that for $i\geq1$
\[
f_{2}^{(1+i)}(0)=(-1)^{i}(1+c\tau^{-i})\,.
\]
and thus $h_{m,k}$ is defined as 

\begin{equation}
\frac{\sum_{l=0}^{m-1}\sum_{p=0}^{l}\sum_{q=0}^{2m-2-l}\hat{a}_{klmpq}\mathbb{B}_{2m-2-l-q,m-1-p}(F_{m-l-q+p})\mathbb{B}_{l,p}^{(q)}(F_{l-p+1})}{\left(-(1+c\tau^{-1})\right)^{2m+k-1}}\,,\label{eq:hmk}
\end{equation}
 where 
\[
F_{n}=\{(-1)^{n}(1+c\tau^{-n})\}_{i=1}^{n}\,,
\]
\begin{eqnarray*}
\hat{a}_{klmpq} & = & \frac{(-1)^{p}(k+m)\Gamma(m)\Gamma(m-p)\Gamma(k+m+p)}{q!l!\Gamma(2+k+l)\Gamma(-l+m)\Gamma(-1-l+2m-q)}\,,
\end{eqnarray*}
\[
\sigma=\frac{1}{\mbox{Ln}\gamma_{2}}\,,\, c=\left(\frac{\beta_{2}}{\gamma_{2}}\right)\left(\frac{\gamma_{2}\tau}{\mbox{Ln}\gamma_{2}}\right)^{\left(\frac{\tau-1}{\tau}\right)}\,,\,\mu=\frac{\ln\tau+\ln(\nicefrac{1}{\mbox{Ln}\gamma_{2}})}{\mbox{Ln}\gamma_{2}}-c\,,
\]
and recalling that $\mathbb{B}_{l,p}^{(q)}$ is defined recursively
in (\ref{eq:ddq Blp}). 

The full description of the $j$th branch root for the characteristic
equation in (\ref{eq:schar}) is therefore 
\[
s_{j}=\frac{1}{\tau}(\ln_{j}\gamma_{2}+\ln(\frac{1}{\ln_{j}\gamma_{2}})+\ln\tau+v_{j})+\alpha_{1}\,,
\]
 where $\ln_{j}(z)=|z|+\arg(z)+2\pi ij$, $v$ is defined in (\ref{eq:veqnsum-1}),
and the subscript on the $v$ indicates the chosen branch for the
logarithms in $\sigma$, $c$, and $\mu$.

For the convenience of the reader, we have included the Mathematica
code in Appendix \ref{sec:Mathematica-Code} which solves for $s_{j}$
as a function of $\alpha$, $\beta$, $\gamma$, $\tau_{1}$, $\tau_{2}$,
$j$, $m$, and $k$. Lastly, upon request, the author will supply
versions of the code with arguments $\alpha_{1}$, $\beta_{2}$, $\gamma_{2}$,
$\tau$, $j$, $m$, and $k$ as well as versions which obtain significant
speedups by using Mathematica's memoization, compilation, and parallelization
features.

\section{\label{sec:Numerical-Results}Numerical Results}

In what follows, we denote $\hat{s}_{j}$ as the actual root. Where
relevant, the $s_{j}$ values were computed using the damped Newton's
method implemented in Mathematica's \textbf{FindRoot} applied to equation
(\ref{eq:schar}).

\subsection{Convergence}

In this section, we provide numerical results for the estimation of
the eigenvalues. Consider the two lag DDE
\begin{eqnarray}
\dot{x}(t) & = & \alpha x(t)+\beta x(t-\tau_{1})+\gamma x(t-\tau_{2})\,;\, t>0\,,\label{eq:DDE eqn redux}\\
x(t) & = & \phi(t)\,;\, t\leq0\,,\nonumber 
\end{eqnarray}
 system from above with transfer function 
\begin{equation}
(s-\alpha-\beta e^{-s\tau_{1}}-\gamma e^{-s\tau_{2}})^{-1}\,.\label{eq:transfer function}
\end{equation}
We note that the accuracy of our asymptotic expansion rests on the
assumption that (\ref{eq:assump 2}) is small and so we choose 
\begin{equation}
\beta_{2}=0.001\,,\,\gamma_{2}=6\,,\,\tau=4\,,\,\alpha_{1}=-1\,.\label{eq:used pars 2}
\end{equation}
which means that the parameters in (\ref{eq:DDE eqn redux}) are 
\begin{equation}
\alpha=-1\,,;\,\beta\approx3.68\times10^{-4}\,,\,\gamma=0.1989\,,\,\tau_{1}=1\,,\,\tau_{2}=5\,,\label{eq:used pars orig}
\end{equation}
and thus 
\[
\frac{\beta_{2}}{\gamma_{2}^{1/\tau}\ln\gamma_{2}}\approx3.6\times10^{-4}\,.
\]
Numerical investigations suggested that the series converges in $m$
about a factor of four times faster than it converges in $k$. Accordingly,
Figure \ref{fig:convergence in n} depicts the convergence in $m$
with $k=m^{4}$. We note that up to $m=8$, the series converges nicely.
For $m>8$, however, the series exhibits a reduction of accuracy,
likely due to the numerical issues inherent in evaluating such large
series. 
\begin{figure}
\begin{centering}
\includegraphics{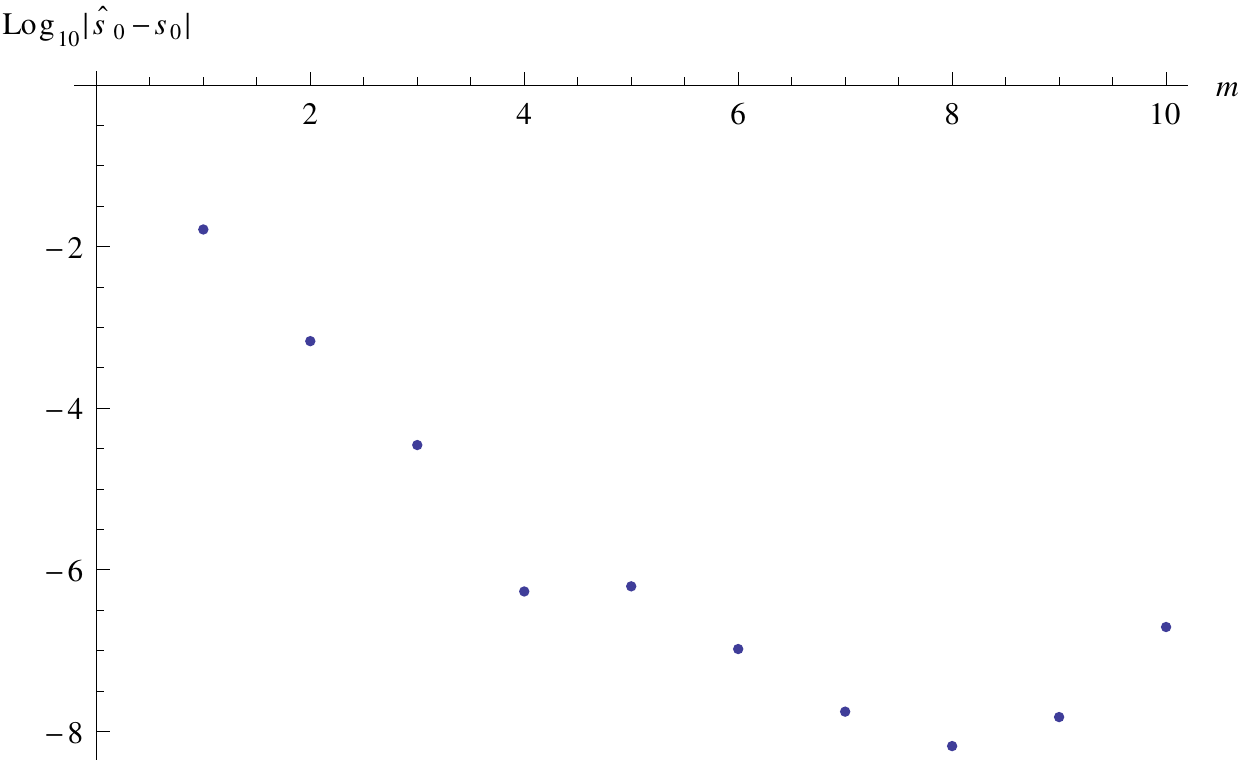}\caption{\label{fig:convergence in n}Plot of $\log_{10}|\hat{s}_{0}-s_{0}|$
for the series approximation to $\hat{s}_{0}$ as a function of $m$
where $k=m^{4}$. The parameters for the two-lag DDE were those from
(\ref{eq:used pars orig}) .}

\par\end{centering}

\end{figure}

Lastly, in subsequent computations of $s_{j}$, we choose to truncate
the series at $m=8$ and $k=1000$ to maintain accuracy. We obtain
the following roots as depicted in Figure \ref{fig:computed eval roots}.
For comparison, Figure \ref{fig:actual evals transfer function} depicts
a contour plot of the transfer function in (\ref{eq:transfer function})
(lighter colors are larger values) and we note that our computed roots
match up nicely with the singularities of the transfer function. 
\begin{figure}
\begin{centering}
\includegraphics[height=2.8in]{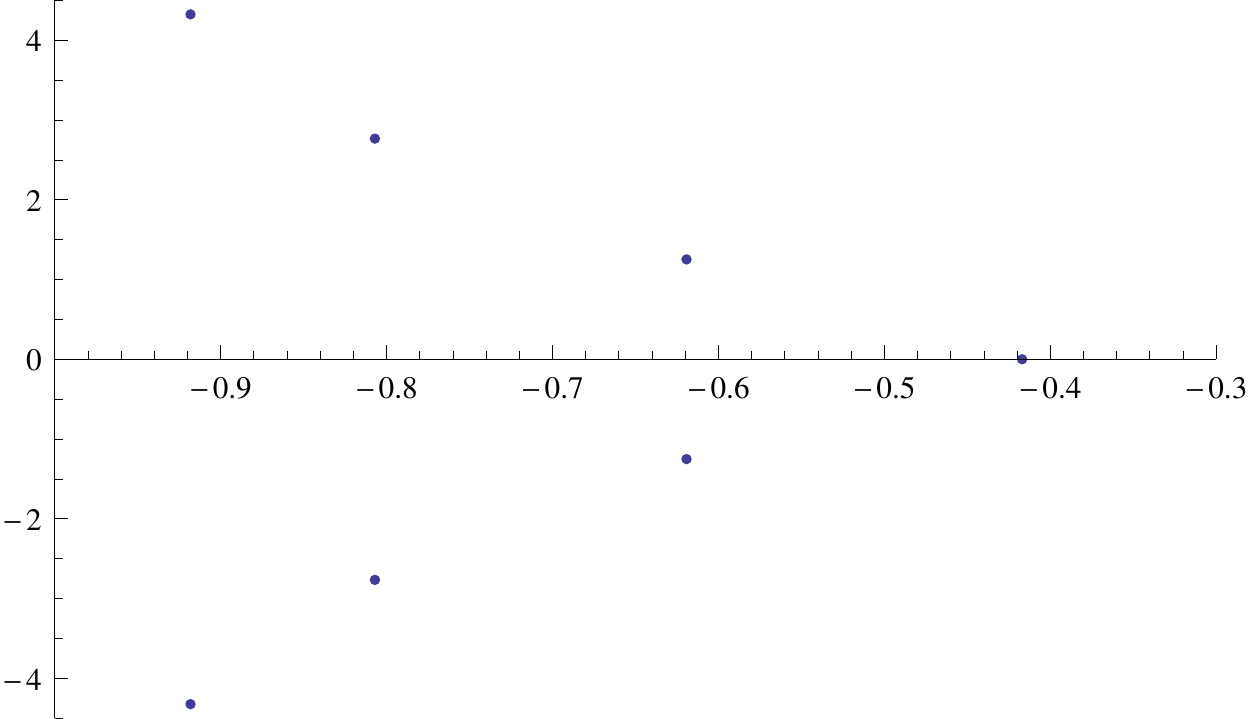}
\par\end{centering}

\caption{\label{fig:computed eval roots}Eigenvalues $s_{j}$ computed using
the expansion in Section (\ref{sec:Two-Lag}). The parameters for
the two-lag DDE were those from (\ref{eq:used pars orig}) .}
\end{figure}
\begin{figure}
\centering{}\includegraphics[height=2.8in]{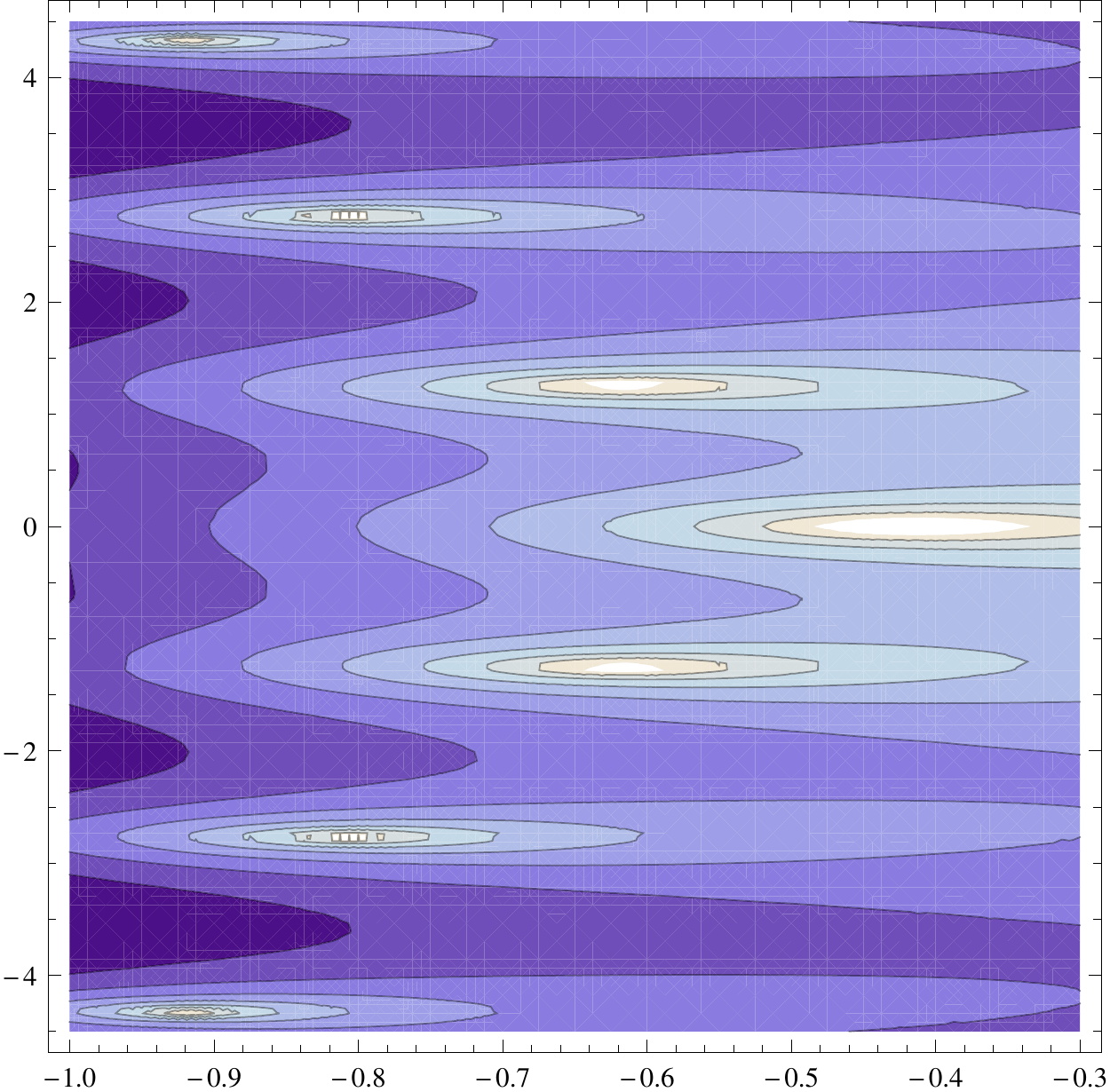}\caption{\label{fig:actual evals transfer function}Contours of $\log_{10}|(s-\alpha-\beta e^{-s\tau_{1}}-\gamma e^{-s\tau_{2}})^{-1}|$
for $c\in\mathbb{C}$. The horizontal axis is the real part of $s$
and the vertical axis is the imaginary part of $s$. Note how the
peaks (lighter regions) coincide with the computed roots in Figure
\ref{fig:computed eval roots}. The parameters for the two-lag DDE
were those from (\ref{eq:used pars orig}).}
\end{figure}

\subsection{\label{sub:Example}Example}

In modeling the life cycle of the Australian blowfly \emph{Lucila
cuprina}, Braddock and van den Driessche \cite{Braddock1983} developed
the following two-lag DDE
\begin{equation}
\frac{dx}{dt}=rx(t)(1-a_{1}x(t-\tau_{1})-a_{2}x(t-\tau_{2}))\,.\label{eq:blowflydde}
\end{equation}
 The positive equilibrium is at $x^{*}=1/(a_{1}+a_{2})$ and a linearization
about this point (shifting to $y=x-x^{*}$) yields the equation
\begin{equation}
\frac{dy}{dt}\approx-rx^{*}a_{1}y(t-\tau_{1})+-rx^{*}a_{2}y(t-\tau_{2})\,.\label{eq:linearized y}
\end{equation}
Theorem 6 in Ruan \cite{Ruan2006} summarizes the results of \cite{Braddock1983}
nicely, providing 3 different sufficiency conditions for the existence
of a value of $\tau_{2}$ which will incur a Hopf bifurcation. We
present here only the last of the possible sufficiency conditions
in the following theorem from \cite{Ruan2006}, since these restrictions
also satisfy our conditions for accurate series expansion.
\begin{theorem}
\label{thm:hopfbifurcation}If $a_{1}=a_{2}$ and $\tau_{1}>(2x^{*}ra_{1})^{-1}$
then there is a $\tau_{2}^{0}>0$ such that when $\tau_{2}=\tau_{2}^{0}$
the two-lag DDE (\ref{eq:blowflydde}) exhibits a Hopf bifurcation
at the equilibrium $x^{*}=1/(a_{1}+a_{2})$.
\end{theorem}

To illustrate the match between our expansion and this theory, let
$r=1$, $a_{1}=a_{2}=1$, $\tau_{1}=10$. Figure \ref{fig:Re s0}
depicts the principal root as a function of $\tau_{2}$ and we note
the root at $\tau_{2}=0.379414$. We also note that the relation which
must be small (\ref{eq:assump 2}) is on the order of $10^{-17}$
for this set of parameters. Figure \ref{fig:Plot-of-stable} depicts
asymptotically stable oscillations (solid curve for $\tau_{2}=0.2$)
and unstable oscillations (dashed curve for $\tau_{2}=1.3$). To better
illustrate the stable and unstable behaviors, we arbitrarily choose
an initial history of $\phi(t)=1\,;\, t\leq0$ for the stable solution
and an initial history of $\phi(t)=\nicefrac{1}{10}\,;\, t\leq0$
for the unstable solution. 
\begin{figure}
\begin{centering}
\includegraphics{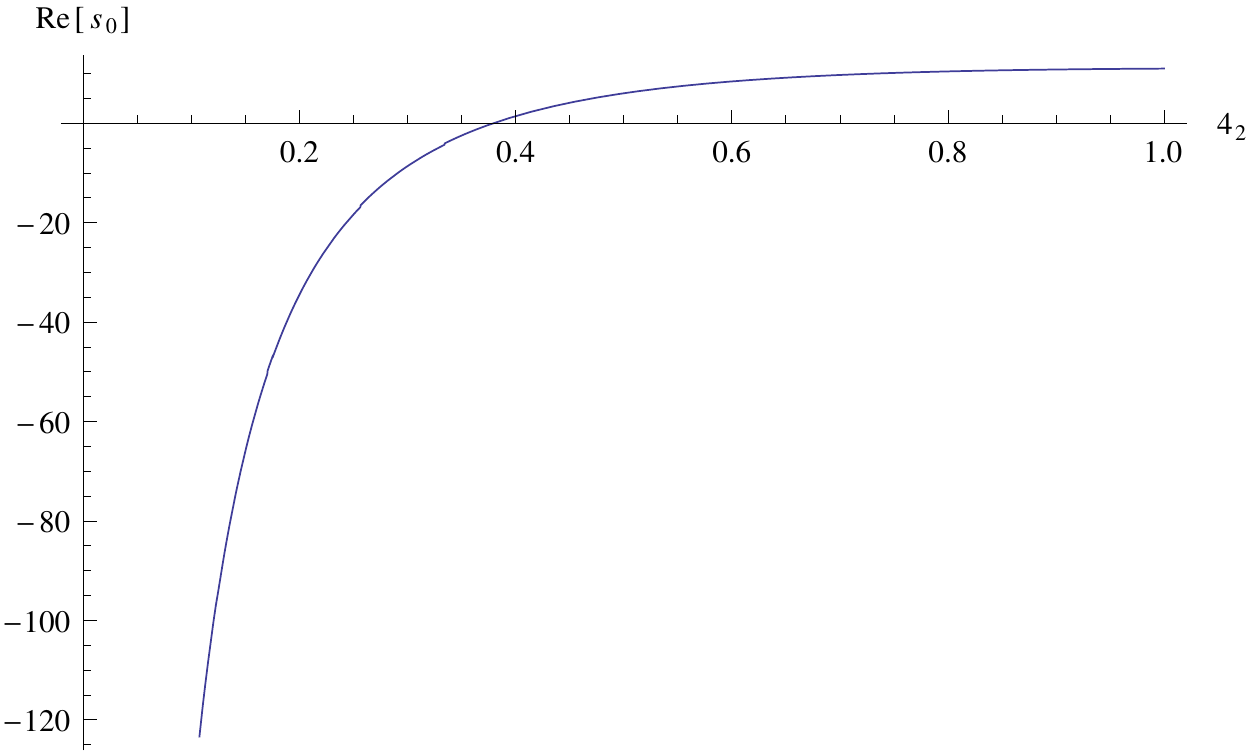}\caption{\label{fig:Re s0}Plot of the $\mbox{Re}[s_{0}]$ as a function of
$\tau_{2}$ for the linearized two-lag blowfly model in (\ref{eq:linearized y}).}

\par\end{centering}

\end{figure}
\begin{figure}
\begin{centering}
\includegraphics{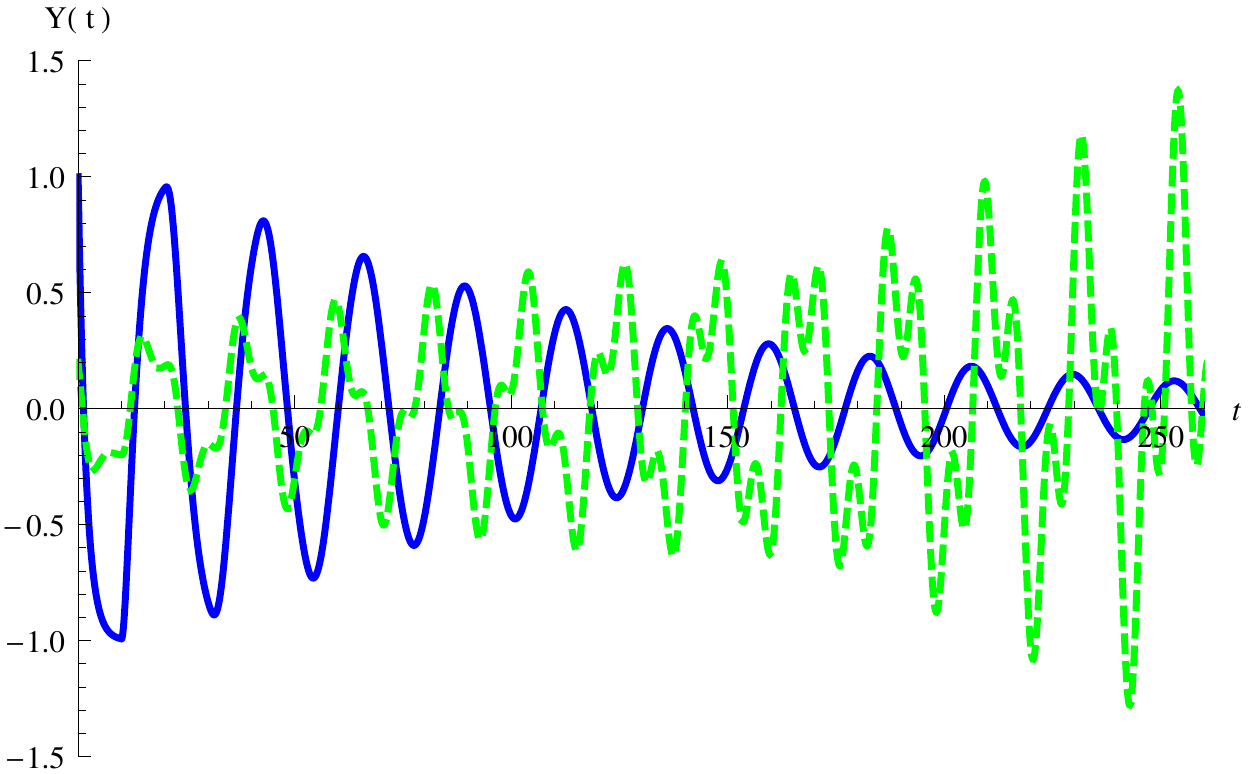}\caption{\label{fig:Plot-of-stable}Plot of stable (solid curve for $\tau_{2}=0.2$)
and unstable (dashed curve for $\tau_{2}=1.3$) solutions to the linearized
two-lag blowfly model in (\ref{eq:linearized y}). For the stable
solution, the initial history is $\phi(t)=1$ for $t\leq0$ and for
the unstable solution, the initial history is $\phi(t)=0.1$ for $t\leq0$.}

\par\end{centering}

\end{figure}

\section{\label{sec:Conclusion-and-Discussion}Conclusion and Discussion}

In this work we have developed an asymptotic expansion for the roots
to a nonlinear eigenvalue problem associated with two-lag DDEs. We
have provided numerical evidence supporting the accuracy of our expansion.
We have also applied the expansion to a two-lag DDE model of the Australian
blowfly, illustrating the Hopf bifurcation transition from stable
to unstable oscillations.

There are many intriguing directions one could take with this work.
Concerning systems of single-lag DDEs, it is known that the approach
presented in Asl and Ulsoy's original article \cite{Asl2003} is only
valid when the coefficient matrices in front of the state and lagged
state commute. This limitation has been discussed in the literature
\cite{Asl2007,Zafer2007} with Jarlebring \cite{Jarlebring2007} also
noting that the matrices need only be triangularizable. Fortunately,
there is an algorithm for working around this problem and the monograph
\cite{Yi2010} has a very clear description of how to solve single-lag
DDE systems using a \emph{matrix }Lambert W function. And a reasonable
extension would be to develop the framework to identify eigenvalues
for two-lag systems of DDEs. Indeed, we have made preliminary efforts
in this direction as it is a natural extension to some of the author's
previous work \cite{BanksBortzHolte2003mbs}.

Another direction for future work would be to consider more general
multi-lag DDEs. We were surprised to discover that the restrictions
placed on the coefficients in the two-lag DDE were satisfied for many
of the DDEs we considered. We therefore plan to investigate if the
coefficient restrictions for an asymptotic expansion in the three-lag
(and higher) DDEs are as generous.

\section{Acknowledgements}

The author wishes to thank Andrew Christlieb, Anna Gilbert, and Peter
Petre for the original discussion which led to the motivating question
for this paper. Portions of this work were performed with support
from DARPA, AFRL, and HRL, under contract FA8650-11-C-7158.

\section{Disclaimer}

The views expressed are those of the author and do not refl{}ect the
offi{}cial policy or position of the Department of Defense or the
U.S. Government. This is in accordance with DoDI 5230.29, January
8, 2009. 

\appendix

\section{\label{sec:Mathematica-Code}Mathematica Code}

Here we present the Mathematica code which was used to generate the
roots $s_{j}$:

\begin{Verbatim}[frame=single]
(*DM Bortz*)
(*dmbortz@colorado.edu*)
(*Mathematical Biology Group*)
(*http://mathbio.colorado.edu*)
(*Applied Mathematics Department*)
(*University of Colorado*)
(*Boulder, CO 80309-0526*)
(*June 2012*)

Clear[Logn, f2p, z, j, w, tau1, tau2, alpha, beta, gamma, k, l, m, n]
Clear[p, q, mmax, kmax, DBellYq, sigma, c, mu, sj, sjMem]
Logn[z_,j_] := Log[Abs[z]] + Arg[z] + 2*Pi*I*j;
f2p[w_, tau1_, tau2_, j_, c_] := (-1)^j*(E^(-w) + 
   c*(1/(E^(w*(tau1/tau2))*(tau2/tau1)^j)));
DBellYq[l_, p_, q_, tau1_, tau2_, c_] := 
   If[l < 0 || p < 0 || (p == 0 && l > 0) || l < p, 0,
      If[q == 0,
         BellY[l, p, Array[f2p[0,tau1,tau2,#1,c]&,{l-p+1}]],
         Sum[Binomial[l, r] * 
            Sum[Binomial[q - 1, s] * 
               DBellYq[l-r,p-1,q-1-s,tau1,tau2,c] * 
               f2p[0, tau1, tau2, 1 + r + s, c],
            {s, 0, q - 1}],
         {r, 1, l - p + 1}]
      ]
   ];
sigma[alpha_,gamma_,tau1_,tau2_,j_] := 
   1/Logn[gamma*tau1*Exp[(-alpha)*tau2],j];
c[alpha_, beta_, gamma_, tau1_, tau2_, j_] := 
   (beta*(Exp[alpha*(tau2 - tau1)]/gamma)) * 
   (gamma*tau2*(Exp[(-alpha)*tau2] / 
      Logn[gamma*tau1*Exp[(-alpha)*tau2], j]))^(1 - tau1/tau2);
mu[alpha_, beta_, gamma_, tau1_, tau2_, j_] := 
   (Log[tau2/tau1] + Log[1/Logn[gamma*tau1*Exp[(-alpha)*tau2], j]]) / 
   Logn[gamma*tau1*Exp[(-alpha)*tau2], j] - 
   c[alpha, beta, gamma, tau1, tau2, j];
sj[alpha_, beta_, gamma_, tau1_, tau2_, j_, mmax_, kmax_] :=     
   (Logn[gamma*tau1*Exp[(-alpha)*tau2], j] + Log[tau2/tau1] - 
    Log[Logn[gamma*tau1*Exp[(-alpha)*tau2], j]] +
    Sum[Sum[(Sum[Sum[Sum[
      (-1)^p Pochhammer[m,k]*((k + m)*Gamma[m]*Gamma[m-p] * 
      Gamma[k+m+p])/(l!*q!*Gamma[2+k+l]*Gamma[-l+m] * 
      Gamma[-1 - l + 2*m - q]) * 
       BellY[2*m - 2 - l - q, m - 1 - p,
          Array[f2p[0,tau1,tau2,#1,c[alpha,beta,gamma,tau1,tau2,j]]&,
          {Max[0, m - l - q + p]}]] * 
          DBellYq[l,p,q,tau1,tau2,c[alpha,beta,gamma,tau1,tau2,j]],
       {q, 0, 2*m - 2 - l}], {p, 0, l}], {l, 0, m - 1}] / 
          (f2p[0, tau1, tau2, 1, 
             c[alpha, beta, gamma, tau1, tau2, j]]^(2*m + k - 1))) * 
          mu[alpha, beta, gamma, tau1, tau2, j]^m * 
          (sigma[alpha, gamma, tau1, tau2, j]^k/m!/k!),
   {m, 1, mmax}], {k, 0, kmax}]) / (tau2/tau1) + 
   alpha*tau1;
\end{Verbatim}

\bibliographystyle{abbrv}
\bibliography{library}

\end{document}